\documentclass{article}

\usepackage{amsmath,amssymb}

\newcommand{\QCB}{\mathsf{QCB}}
\newcommand{\QCBZ}{\mathsf{QCB_0}}
\newcommand{\Seq}{\mathsf{Seq}}
\newcommand{\omegaScottTop}{\tau_{\omega\mathrm{Scott}}}
\newcommand{\ScottTop}{\tau_{\mathrm{Scott}}}
\newcommand{\COTop}{\tau_{\mathrm{CO}}}
\newcommand{\omegaCOTop}{\tau_{\omega\mathrm{CO}}}
\newcommand{\CC}{\mathsf{C}} 
\newcommand{\calF}{\mathcal{F}}
\newcommand{\calK}{\mathcal{K}}
\newcommand{\calO}{\mathcal{O}}
\newcommand{\INu}{{\mathbb{N}_\infty}}
\newcommand{\IN}{\mathbb{N}}
\newcommand{\IS}{\mathbb{S}}
\newcommand{\Cls}{\mathit{Cls}}
\newcommand{\cf}{\mathit{cf}}
\newcommand{\KuV}{{\mathcal{K}_\mathrm{uV}}}
\newcommand{\calOOp}{{\mathcal{O\!O}_{\!+}}}

\newtheorem{lemma}{Lemma}[section]
\newtheorem{theorem}[lemma]{Theorem}
\newtheorem{corollary}[lemma]{Corollary}
\newtheorem{proposition}[lemma]{Proposition}
\newtheorem{example}[lemma]{Example}
\newenvironment{proof}{{\noindent\bf Proof.}}{\hspace*{\fill}$\Box$\par\bigskip}

\renewcommand{\labelenumi}{{\rm(\theenumi)}}

\begin{document}

\title{A Hofmann-Mislove Theorem for Scott open sets}

\author{Matthias Schr\"oder, TU Darmstadt}

\date{\today}

\maketitle

\begin{abstract}
We consider the intersection map 
on the family of non-empty $\omega$-Scott-open sets of the lattice of opens 
of a topological space.
We prove that in a certain class of topological spaces
the intersection map forms a continuous retraction 
onto the space of countably compact subsets of the space
equipped with (the sequentialisation of) the upper Vietoris topology.
This class consists of all sequential spaces which are sequentially Hausdorff.
\end{abstract}

\section{Introduction} \label{sec:intro}

The Hofmann-Mislove Theorem states
that in a sober space $X$ the intersection map $\calF \mapsto \bigcap(\calF)$ is a bijection between 
the non-empty Scott-open filters on the lattice of open subsets of $X$
and the compact saturated subsets of $X$ (cf.\ \cite{GH80}). 
The natural question arises in which cases (and how) this result can be generalised
to all non-empty Scott-open collection of opens.

We give a positive answer for sequentially Hausdorff sequential spaces.
Remember that a space is \emph{sequentially Hausdorff}, if every convergent sequence has a unique limit.
Our version of the Hofmann-Mislove Theorem states 
that the intersection map 
defined on the family $\calOOp(X)$ of all non-empty $\omega$-Scott-open collections of opens
is a continuous retraction onto the space $\calK(X)$ of all countably-compact subsets in such spaces $X$.
The corresponding topologies are the $\omega$-Scott topology on $\calOOp(X)$
and the sequentialisation of the upper Vietoris topology on $\calK(X)$.
For the proof we only use Dependent Choice (DC), 
whereas the classical Hofmann-Mislove Theorem needs the Axiom of Choice (AC).

Beforehand we summerise in Section~\ref{sec:prelim} some properties of sequential spaces,
of the $\omega$-Scott topology on open subsets, 
of the upper Vietoris topology on countably compact sets,
and of sequential Hausdorffness.


\section{Sequential spaces, hyperspaces, sequentially Hausdorff spaces}\label{sec:prelim}

\subsection{Sequential spaces and qcb-spaces}

A \emph{sequential space} is a topological space in which every sequentially open subset is open
(cf.\ \cite{En89}.
The category $\Seq$ of sequential spaces and continuous functions as morphisms is cartesian closed.
The following lemma explains how function spaces $Y^X$ are formed in $\Seq$ 
(see e.g. \cite{ELS04, Sch:phd} for details).

\begin{lemma}\label{l:funcspaces:Seq}
 Let $X,Y$ be sequential spaces. 
 Then the function space $Y^X$ formed in $\Seq$ 
 has the set $C(X,Y)$ of continuous functions from $X$ to $Y$
 as underlying set and its topology is equal to both
 \begin{itemize}
  \item[(a)] 
    the sequentialisation of the compact-open topology on $C(X,Y)$ and
  \item[(b)]  
    the sequentialisation of the countably-compact-open topology on $C(X,Y)$.
 \end{itemize}
 The convergence relation of $Y^X$ is continuous convergence, 
 i.e., $(f_n)_n$ converges to $f_\infty$ in $Y^X$
 iff $(f_n(x_n))_n$ converges to $f_\infty(x_\infty)$ in $Y$
 whenever $(x_n)_n$ converges to $x_\infty$ in $X$.
 Equivalently, $(f_n)_n$ converges to $f_\infty$ in $Y^X$ iff 
 the transpose $\hat{f}:\INu \times X \to Y$ defined by $\hat{f}(n,x)=f_n(x)$
 is continuous, where $\INu$ denotes the one-point compactification of $\IN$.
\end{lemma}

\noindent
Remember that the sequentialisation of a topology is 
the family of all sequentially open sets pertaining to that topology.
Binary and countable products in $\Seq$ are formed 
by taking the sequential coreflection (sequentialisation) of the usual Tychonoff product. 
We denote the binary product in $\Seq$ by $X \times Y$.
 
A \emph{qcb-space} is a quotient of a countably based space.
The category $\QCB$ of qcb-spaces forms a cartesian closed subcategory of $\Seq$;
countable products and function spaces are inherited from $\Seq$ (cf. \cite{ELS04,Sch:phd}).
The same is true for the full category $\QCBZ$ of qcb-spaces that satisfy the $T_0$-property.


\subsection{Topologies on open subsets}\label{sub:open}

For a given sequential space $X$ let $\calO(X)$ denote the family of open subsets of $X$.
On $\calO(X)$ we consider four natural topologies.
For us, the most important topology is the $\omega$-Scott topology $\omegaScottTop$.
A subset $H \subseteq \calO(X)$ is called \emph{$\omega$-Scott open},
if $H$ is upwards closed in the complete lattice $(\calO(X);\subseteq)$ and $\mathcal{D}\cap H \neq \emptyset$
for each countable directed subset $\mathcal{D}$ with $\bigcup\mathcal{D}\in H$.
The second topology is the Scott topology $\ScottTop$ on $(\calO(X);\subseteq)$
which is defined like the $\omega$-Scott topology, except for considering all directed families of opens,
not only all countable ones.
Clearly, the $\omega$-Scott-topology refines the Scott topology.

The third topology on $\calO(X)$ is the \emph{compact-open topology} $\COTop$.
A basis of $\COTop$ is given by sets of the form $K^\subseteq :=\{U\in\calO(X) \,|\, K\subseteq U\}$,
where $K$ runs through the compact subsets of $X$.
The name ``compact-open topology'' is motivated by the fact that
it coincides with the usual compact-open topology on the function space $\IS^X$,
under the natural identification of an open subset $U\subseteq X$ with its continuous
characteristic function $\cf(U)$.
Here $\IS$ denotes the Sierpi{\'n}ski space which has $\{\top,\bot\}$ as its underlying set
and in which $\{\top\}$, but not $\{\bot\}$ is open.
The map $\cf(U) \colon X \to \IS$ is defined by $\cf(U)(x)=\top :\Longleftrightarrow x \in U$.
Obviously, $\cf\colon U \mapsto \cf(U)$ constitutes a bijection between $\calO(X)$ and $\IS^X$.

The fourth topology on $\calO(X)$ is the \emph{countably-compact-open topology} $\omegaCOTop$.
It is defined like the compact-open topology, except for considering all countably compact subsets $K$.
Remember that $K \subseteq X$ is called \emph{countably compact}, 
if every countable open cover of $K$ contains a finite subcover of $K$.

Clearly, $\COTop\subseteq\ScottTop \subseteq \omegaScottTop$ and $\COTop\subseteq \omegaCOTop \subseteq \omegaScottTop$. 
For hereditarily Lindel{\"of} spaces (and thus for qcb-spaces)
we have $\ScottTop=\omegaScottTop$ and $\COTop = \omegaCOTop$.

Now we present some properties of the convergence relations induced by these topologies.
Most of the statements of Proposition~\ref{p:calO} belong to the folklore of sequential space,
nevertheless it seems worth collecting their proofs here.

\begin{proposition}\label{p:calO} 
 Let  $X$ be a sequential space.
\begin{enumerate}
 \item \label{i:calO:convrel:equal}
  The convergence relations induced on the set $\calO(X)$
  by the Scott topology, 
  by the $\omega$-Scott topology,
  by the compact-open topology,
  and by the countably-compact-open topology 
  coincide.
 \item \label{i:calO:cup:cap}
  A sequence $(U_n)_n$ converges to $U_\infty$ with respect to the compact-open topology
  if, and only if,
  \[
    W_k:=\bigcap\limits_{n \geq k}U_n \cap U_\infty \in\calO(X)  \text{ for all $k \in \IN$}
    \quad\text{and}\quad
    U_\infty=\bigcup\limits_{k \in \IN}W_k.
  \]
 \item \label{i:calO:convrel}
  A sequence $(U_n)_n$ converges to $U_\infty$ w.r.t.\ the compact-open topology  
  if, and only if,
  for every sequence $(x_n)_n$ converging in $X$ to some element in $U_\infty$
  there is some $n_0 \in \IN$ such that $x_n \in U_n$ for all $n \geq n_0$.  
\item \label{i:calO:sequentialisation}
  The sequentialisation
  of each of the four topologies mentioned in \eqref{i:calO:convrel:equal} is equal to $\omegaScottTop$.    
 \item \label{i:calO:homeomorphic}
  The space $\calO(X)$ equipped with the $\omega$-Scott topology is homeomorphic to $\IS^X$
  via the map $\cf$.
 \item \label{i:calO:qcb}
  If $X$ is a qcb-space, then the Scott topology on $\calO(X)$ is a qcb$_0$-topology.
\end{enumerate}
\end{proposition}

\begin{proof}
\renewcommand{\labelenumi}{{\rm\alph{enumi})}}
\begin{enumerate}
\item
 To show the only-if-parts of~\eqref{i:calO:cup:cap} and \eqref{i:calO:convrel},
 let $(U_n)_n$ converge to $U_\infty$ w.r.t.\ the compact-open topology $\COTop$.
 Then the sequence $(\cf(U_n))_n$ of characteristic functions
 converges to $\cf(U_\infty)$ with respect to the compact-open topology on $\IS^X$.
 By Lemma~\ref{l:funcspaces:Seq}, $(\cf(U_n))_n$ converges continuously to $\cf(U_\infty)$.
 Therefore the universal characteristic function $u\colon \INu \times X \to \IS$
 mapping $(n,x)$ to $\cf(U_n)(x)$ is sequentially continuous.
 As $\{\infty,n \,|\, n\geq k \in \IN\}$ is sequentially compact in $\INu$ for each $k$ 
 and $\{\top\}$ is open, the set
 \[
   W_k= \big\{ x \in X \,\big|\, \forall k \leq n \leq \infty.\, u(n,x) \in \{\top\} \big\}
 \]
 is open in $X$ (see Lemma 2.2.9 in \cite{Sch:phd}).
 Clearly $W_k=\bigcap_{n \geq k} U_n \cap U_\infty$.
 By the continuity of $u$, for any $x \in U_\infty$ there is some $k_x$ such $u(n,x)=\top=u(\infty,x)$
 for all $n \geq k_x$, implying $x \in W_{k_x}$.
 So $(U_n)_{n \leq \infty}$ fulfils the displayed formula in \eqref{i:calO:cup:cap}. 

 Moreover, if $(x_n)_n$ converges to $x \in U_\infty$, then there is some $n_1 \in \IN$ 
 such that $x_n \in W_{k_x}$ for every $n \geq n_1$, 
 implying $x_n \in U_n$ for all $n \geq \max\{k_x,n_1\}$.
 This shows the only-if-part of~\eqref{i:calO:convrel}. 
\item
 Now we prove that the right hand side of \eqref{i:calO:cup:cap}  
 implies that $(U_n)_n$ converges to $U_\infty$ w.r.t.\ the $\omega$-Scott topology $\omegaScottTop$.
 Let $H$ be an $\omega$-Scott-open set containing $U_\infty$.
 Since $(W_k)_k$ is an increasing sequence of opens with $U_\infty=\bigcup_{k \in \IN} W_k$,
 there is some $k_0$ such that $W_{k_0} \in H$, implying $U_n \in H$ for all $n \geq k_0$,
 as required.
\item
 For any countably compact set $K$ the set $K^\subseteq$ is obviously an element of $\omegaScottTop$. 
 Thus $\omegaScottTop$ refines $\omegaCOTop$, 
 which in turn refines $\COTop$, because compact subsets are countably compact.
 So convergence with respect to $\omegaScottTop$
 implies convergence w.r.t.\ $\omegaCOTop$,
 and convergence w.r.t.\ $\omegaCOTop$ implies convergence w.r.t.\ $\COTop$.
 Furthermore $\omegaScottTop$ refines $\ScottTop$,
 which in turn refines $\COTop$, because $K^\subseteq$ is Scott-open for every compact subset $K$ of $X$.
 So convergence w.r.t.\ $\omegaScottTop$
 implies convergence w.r.t.\ $\ScottTop$,
 and convergence w.r.t.\ $\ScottTop$ implies convergence w.r.t.\ $\COTop$.
 \\
 We conclude that all four topologies in \eqref{i:calO:convrel:equal}
 induce the same convergence relation on $\calO(X)$.
 This completes the proofs of~\eqref{i:calO:convrel:equal} and \eqref{i:calO:cup:cap}.
\item
 For a sequence $(U_n)_{n \leq \infty}$  of open sets, 
 the right hand side of Statement~\eqref{i:calO:convrel} is equivalent to saying
 that $(\cf(U_n))_n$ converges continuously to $\cf(U_\infty)$, 
 because in the Sierpinski space $\IS$ every sequence converges to $\bot$.
 By Lemma~\ref{l:funcspaces:Seq}
 the convergence relation induced by the compact-open topology on $\IS^X$ is continuous convergence.
 This proves the if-part of~\eqref{i:calO:convrel}. 
\item
 Let $H$ be a sequentially open in the topological space $(\calO(X),\omegaScottTop)$.
 Then $H$ is upwards closed, as $U \subseteq V$ and $U,V \in \calO(X)$ implies that
 the constant sequence $(V)_n$ converges to $U$ w.r.t.\ $\omegaScottTop$.
 Now let $(U_n)_n$ be an increasing sequence of open sets with $U_\infty:=\bigcup_{n \in \IN} U_n \in H$.
 Then $(U_n)_n$ converges to $U_\infty$ w.r.t.\ the $\omega$-Scott topology, 
 because for every $\omega$-Scott open set $G$ containing $U_\infty$ 
 there is some $m \in \IN$ such that $U_m \in G$, 
 implying $U_n \in G$ for all $n \geq m$, because $G$ is upwards closed.
 As $H$ is sequentially open, we have $U_n \in H$ for almost all $n$.
 We conclude that $H$ is $\omega$-Scott-open and the topology $\omegaScottTop$ is sequential.
\item
 As the four topologies induce the same convergence relation as $\omegaScottTop$
 by~\eqref{i:calO:convrel:equal} and $\omegaScottTop$ is sequential,
 the sequentialisation of either of them is equal to $\omegaScottTop$.
 This shows~\eqref{i:calO:sequentialisation}.
\item
 The bijection $\cf\colon \calO(X) \to \IS^X$ and its inverse are both
 continuous functions w.r.t.\ the corresponding compact-open topologies 
 on the sets $\calO(X)$ and $\IS^X$.
 Since $\omegaScottTop$ is the sequentialisation of $\COTop$ by \eqref{i:calO:sequentialisation}
 and $\IS^X$ carries the sequentialisation of the compact-open topology on $C(X,\IS)$
 by Lemma~\ref{l:funcspaces:Seq},
 $\cf$ and its inverse are also continuous w.r.t.\ this pair of sequential topologies.
 Hence $\cf$ is a homeomorphism between $\calO(X)$ equipped with the $\omega$-Scott topology and $\IS^X$.
 This shows~\eqref{i:calO:homeomorphic}.
\item
 If $X$ is a qcb-space, then $X$ is hereditarily Lindel\"of (see \cite[Proposition 3.3.1]{Sch:phd}).
 Therefore $\ScottTop=\omegaScottTop$.
 Since $(\calO(X),\omegaScottTop)$ is homeomorphic to $\IS^X$ by~\eqref{i:calO:homeomorphic}, 
 it is a qcb$_0$-space.
 This shows~\eqref{i:calO:qcb}.
\end{enumerate}
\end{proof}

In view of this proposition, we henceforth denote by $\calO(X)$
the sequential space consisting of the opens of $X$ as its underlying set
and carrying the $\omega$-Scott topology $\omegaScottTop$.
Note that the $\omega$-Scott topology forms the underlying set of the sequential space $\calO(\calO(X))$,
which itself carries the $\omega$-Scott topology defined on the complete lattice $(\calO(\calO(X)); \subseteq)$.
Moreover, $\calO(\calO(X))$ is a qcb$_0$-space, if $X$ is a qcb-space 
by Proposition~\ref{p:calO}\eqref{i:calO:qcb}.


\subsection{The upper Vietoris topology on compact sets}\label{sub:compact}

Now we dicuss a natural topology on the set of compact subsets of $X$, 
known as the \emph{upper Vietoris topology}.
In view of the results of Section~\ref{sec:HofmannMislove} 
it makes sense to consider the larger family of countably compact subsets of $X$.
We denote the  latter by $\calK(X)$.
For a qcb$_0$-space $X$, however, $\calK(X)$ coincides with the family of compact subsets,
because $X$ is hereditarily Lindel\"of.

The \emph{upper Vietoris topology} is 
generated by the subbasic open sets $\Box U := \{K\in \calK(X) \,|\, K\subseteq U\}$,
where $U$ varies over the open subsets of $X$.
A related topology is the \emph{Vietoris topology}.
It has as a subbasis the sets $\Box U$ and 
$\Diamond U := \{K\in \calK(X) \,|\, K\cap U \neq \emptyset\}$, where again $U\in\calO(X)$. 

We do not know whether any of these topologies is sequential.
However, both topologies are countably based and thus sequential, 
provided that $X$ has a countable basis.
By $\KuV(X)$ we denote the space of countably compact subsets of $X$ 
equipped with the sequentialisation of the upper Vietoris topology.
If $X$ is $T_1$, then $\KuV(X)$ is a $T_0$-space;
if $X$ is a qcb-space, then $\KuV(X)$ is a qcb-space as well (cf.\ Section 4.4.3 in \cite{Sch:phd}).

The countably-compact-open topology on the set of continuous functions from $X$ to $Y$
is defined by the subbasis of open sets $\CC(K,U):=\{f\in Y^X \,|\,f[K]\subseteq U\}$,
where $K$ is countably compact and $U\in\calO(Y)$.
By Lemma~\ref{l:funcspaces:Seq}, $\CC(K,U)$ is open in the exponential $Y^X$ formed in $\Seq$.
We show that this construction is continuous.

\pagebreak[3]
\begin{proposition}\label{p:comp-op}
For two sequential spaces $X,Y$, the map $\CC$ is a sequentially continuous function
from $\KuV(X)\times\calO(Y)$ to $\calO(Y^X)$.
\end{proposition}

\pagebreak[3]
\begin{proof}
 Let $(K_n)_n$ converge to $K_\infty$ in $\KuV(X)$,
 let $(V_n)_n$ converge to $V_\infty$ in $\calO(Y)$,
 let $(f_n)_n$ converge to $f_\infty$ in $Y^X$
 and let $f_\infty \in \CC(K_\infty,V_\infty)$.
 Since $K_\infty \subseteq U_\infty:=f_\infty^{-1}[V_\infty] \in \calO(X)$,
 there is some $n_1$ such that $K_n \subseteq U_\infty$ for all $n \geq n_1$.
 From the fact that $(K_n)_n$ converges to $K_\infty$
 one can easily deduce that the set $L:=K_\infty \cup \bigcup_{i \geq n_1} K_i$ is countably compact in $X$.
 Hence $M:=f_\infty[L] \subseteq V_\infty$ is countably compact in~$Y$.
 By Proposition~\ref{p:calO}\eqref{i:calO:cup:cap} the set
 $W_m:=V_\infty \cap \bigcap_{n \geq m} V_n$ is open in $Y$ for all $m \in \IN$
 and $V_\infty=\bigcup_{m\in \IN} W_m$.
 Therefore there is some $n_2 \in \IN$ with $M \subseteq W_0 \cup \dotsc \cup W_{n_2}$,
 hence $M \subseteq W_{n_2}$ and $f_\infty \in \CC(L,W_{n_2})$.
 Since the countably-compact-open topology induces the convergence relation on $Y^X$
 by Lemma~\ref{l:funcspaces:Seq},
 there is some $n_3 \in \IN$ such that $f_n \in \CC(L,W_{n_2})$ for all $n \geq n_3$.
 For all $n \geq \max\{n_1,n_2,n_3\}$ we have $K_n \subseteq L$, $W_{n_2} \subseteq V_n$
 and thus $f_n \in \CC(L,W_{n_2}) \subseteq \CC(K_n,V_n)$.
 We conclude that $\big(\CC(K_n,V_n)\big)_n$ converges to $\CC(K_\infty,V_\infty)$
 in $\calO(Y^X)$.
 Hence $\CC$ is sequentially continuous.
\end{proof}


\subsection{Sequentially Hausdorff sequential spaces}\label{sub:seqHaus}

We discuss sequentially Hausdorff spaces.
A topological space is called \emph{sequentially Hausdorff}, if any convergent sequence has a unique limit.
Hausdorff spaces are sequentially Hausdorff and sequentially Hausdorff spaces are $T_1$, 
whereas the converse does not hold in either case (cf.\ \cite{En89}). 
Sequentially Hausdorff sequential spaces are called $\mathcal{L}^*$-spaces by R.~Engelking in \cite{En89}.

The following lemma summerises essential properties of sequentially Hausdorff sequential spaces.

\begin{lemma}\label{l:props:seqHaus}
Let $X$ be a sequential space that is sequentially Hausdorff.
\begin{enumerate}
 \item \label{i:neq:continuous}
  The inequality function $\mathit{neq}\colon X \times X \to \IS$
  defined by $\mathit{neq}(x,y)=\top :\Longleftrightarrow x \neq y$ is sequentially continuous.
 \item \label{i:X->O(X)}
  The map $x \mapsto X \setminus \{x\}$ is a continuous function from $X$ to $\calO(X)$.
 \item \label{i:OO(X)->O(X)}
  The map $H \mapsto X \setminus \bigcap(H)$ is a continuous function from $\calO(\calO(X))$ to $\calO(X)$.
 \item \label{i:props:seqHaus:Hnxn}
  Let $(H_n)_n$ converge to $H_\infty$ in $\calO(\calO(X))$
  and let $(x_n)_n$ converge to $x_\infty$ in~$X$.
  If $x_n \in \bigcap(H_n)$ for all $n \in \IN$, then $x_\infty \in \bigcap(H_\infty)$.
\end{enumerate}
\end{lemma}

\begin{proof}
\begin{enumerate}
 \item
  Let $(x_n)_n$ converge to $x_\infty$ and $(y_n)_n$ converge to $y_\infty$ in $X$.
  We have to show that $(\mathit{neq}(x_n,y_n))_n$ converges to $\mathit{neq}(x_\infty,y_\infty)$ in $\IS$.
  The only interesting case is 
  that there exists a strictly increasing function $\varphi\colon \IN \to \IN$
  such $\mathit{neq}(x_{\varphi(n)},y_{\varphi(n)})=\bot$ for all $n \in \IN$. 
  Then we have $x_{\varphi(n)}=y_{\varphi(n)}$ for all $n \in \IN$.
  Since $X$ is sequentially Hausdorff, this implies $x_\infty=y_\infty$, 
  hence $\mathit{neq}(x_\infty,y_\infty)=\bot$, as required.
 \item
  By \eqref{i:neq:continuous} and by cartesian closedness of $\Seq$, 
  the function $s\colon X \to \IS^X$ defined by $s(x)(y):=\mathit{neq}(x,y)$ is continuous.
  Clearly, $s(x)$ is the characteristic function of the open set $X \setminus \{x\}$.  
 \item
  By the cartesian closedness of $\Seq$ and by (1), 
  the function $t\colon \calO(\calO(X)) \to \calO(X)$
  defined by $t(H):=\{ x \in X \,|\, (X \setminus \{x\}) \in H \}$ is continuous.
  If $x \in t(H)$ then clearly $x \notin \bigcap(H)$.
  Conversely, if $x \notin \bigcap(H)$, then there is some $V \in H$ with $x \notin V$,
  implying $X \setminus \{x\} \in H$, as $V \subseteq X \setminus \{x\}$.
  Therefore $t(H) = X \setminus \bigcap(H)$.
  Note that this shows that $\bigcap(H)$ is closed.
 \item
  Let $U_n:=X \setminus \{x_n\}$ for all $n \leq \infty$.
  By (1), $(U_n)_n$ converges to $U_\infty$ in $\calO(X)$. 
  Suppose for contradiction that $x_\infty \notin \bigcap(H_\infty)$.
  Then $U_\infty \in H_\infty$.
  By Proposition~\ref{p:calO}\eqref{i:calO:convrel} there is some $n_0 \in \IN$ 
  such that $U_n \in H_n$ for all $n \geq n_0$, 
  implying $x_{n_0} \notin \bigcap(H_{n_0})$, a contradiction. 
\end{enumerate}
\end{proof}

We remark that sequential Hausdorffness is necessary for all statements in Lemma~\ref{l:props:seqHaus},
taking into account that 
Statement \eqref{i:X->O(X)} comprises that $\{x\}$ is closed for all $x \in X$
and Statement \eqref{i:OO(X)->O(X)} comprises that $\bigcap(H)$ is closed for all $\omega$-Scott open sets $H$.


\section{A version of the Hofmann-Mislove Theorem for $\omega$-Scott open sets} \label{sec:HofmannMislove}

We now formulate and prove our version of the Hofmann-Mislove Theorem for $\omega$-Scott open sets.

We equip the set of non-empty $\omega$-Scott open sets
with the subspace topology inherited from $\calO(\calO(X))$
and denote the resulting space by $\calOOp(X)$.
This space is sequential, because the singleton $\{\emptyset\}$ 
is closed in the space $\calO(\calO(X))$ which is itself sequential 
by carrying the $\omega$-Scott topology (cf.\ Section~\ref{sub:open}).
Recall that $\KuV(X)$ is the space of countably-compact subsets of $X$ 
equipped with the sequentialisation of the upper Vietoris topology 
(cf.\ Section~\ref{sub:compact}).

\begin{theorem} \label{th:K:retract:OO}
 Let $X$ be a sequentially Hausdorff sequential space.
 Then the space $\KuV(X)$ is a continuous retract of $\calOOp(X)$.
 The map $H \mapsto \bigcap(H)$ is a continuous retraction 
 to the continuous section $K \mapsto \{ U \in \calO(X) \,|\, K \subseteq U\}$.
\end{theorem}

\noindent
Note that the condition of $X$ being sequentially Hausdorff is essential,
see Example~\ref{ex:seqHaus:essential}.

We prove Theorem~\ref{th:K:retract:OO} by a series of lemmas.
The following embedding lemma follows in the case of qcb$_0$-spaces from \cite[Proposition 4.4.9]{Sch:phd}.
By a sequential embedding we mean a continuous injection that reflects convergent sequences.

\begin{lemma}\label{l:K:embeds:into:OO}
 For any sequential $T_0$-space $X$ the map $e\colon \KuV(X) \to \calO(\calO(X))$ 
 defined by $e(K):=\{ U \in \calO(X) \,|\, K \subseteq U\}$ is a sequential embedding.
\end{lemma}

\begin{proof}
 By Proposition~\ref{p:calO}\eqref{i:calO:sequentialisation}, $e(K)$ is $\omega$-Scott-open
 for all $K \in \KuV(X)$, hence $e(K) \in \calO(\calO(X))$.
 \\
 To show that $e$ is injective, let $K_1,K_2 \in \calK(X)$ with $K_1 \neq K_2$. 
 W.l.o.g.\ we can assume that there is some $x \in K_1 \setminus K_2$.
 Then $K_2 \subseteq X \setminus \Cls\{x\}$, as $K_2$ is saturated. 
 Hence $X \setminus \Cls\{x\} \in e(K_2) \setminus e(K_1)$. 
 Thus $e$ is injective.
 \\
 To show the continuity of $e$,
 let $(K_n)_n$ converge to $K_\infty$ in $\KuV(X)$
 and let $(U_n)_n$ converge in $\calO(X)$ to an open set $U_\infty \in e(K_\infty)$.
 By Proposition~\ref{p:calO}\eqref{i:calO:cup:cap} 
 and by countable compactness of $K_\infty \subseteq U_\infty$,
 there is some $k \in \IN$ such that $K_\infty \subseteq \bigcap_{i\geq k} U_i \cap U_\infty \in \calO(X)$.
 This implies that $(K_n)_n$ is eventually in $\Box(\bigcap_{i\geq k} U_i \cap U_\infty)$.
 This entails $K_n \subseteq U_n$ and $U_n \in e(K_n)$ for almost all $n$.
 Thus $(e(K_n))_n$ converges to $e(K_\infty)$ in $\calO(\calO(X))$.
 \\
 Conversely, let $(e(K_n))_n$ converge to $e(K_\infty)$ in $\calO(\calO(X))$
 and let $U$ be an open set with $K_\infty \in \Box U$.
 Then $U \in e(K_\infty)$, implying $U \in e(K_n)$ for almost all $n \in\IN$ by Proposition~\ref{p:calO}\eqref{i:calO:convrel}.
 This means that $(K_n)_n$ is eventually in the subbasic set $\Box U$.
 We conclude that $(K_n)_n$ converge to $K_\infty$ in $\KuV(X)$.
\end{proof}

The following lemma is the key observation for showing countable compactness of $\bigcap(H)$.

\begin{lemma}\label{l:TheConvergingSubsequenceLemma}
 Let $X$ be a sequential space. Let $U$ be open and let $(x_n)_n$ be a sequence of $X$.
 Then $(x_n)_n$ has a subsequence that converges to an element in~$U$ 
 if, and only if,
 $(X \setminus \Cls\{x_n\})_n$ does not converge to~$U$ in $\calO(X)$.
\end{lemma}

\begin{proof}
\emph{If-part:}
 If $(X \setminus \Cls\{x_n\})_n$ does not converge to $U$ in $\calO(X)$,
 then by Proposition~\ref{p:calO}\eqref{i:calO:convrel}
 there is a convergent sequences $(a_n)_n \to a_\infty$ in $X$ such that $a_\infty \in U$
 and $a_n \notin X \setminus \Cls\{x_n\}$ for infinitely many $n \in \IN$.
 So there is a strictly increasing function $\varphi\colon \IN \to \IN$  
 with $a_{\varphi(n)} \in \Cls\{x_{\varphi(n)}\}$ for all $n \in \IN$.
 As $(a_{\varphi(n)})_n$ converges to $a_\infty$ and $x_{\varphi(n)}$ is contained in any neighbourhood of $a_{\varphi(n)}$, 
 $(x_{\varphi(n)})_n$ converges to $a_\infty$ as well. 
\\ 
 \emph{Only-if-part:} 
  Let $(x_{\varphi(n)})_n$ be a subsequence of $(x_n)_n$ converging to some element $x_\infty \in U$.
  By Proposition~\ref{p:calO}\eqref{i:calO:convrel}
  neither $(X \setminus \Cls\{x_{\varphi(n)}\})_n$ nor $(X \setminus \Cls\{x_n\})_n$
  converges to $U$ in $\calO(X)$, because $x_{\varphi(n)} \notin X \setminus \Cls\{x_{\varphi(n)}\}$
  for all $n \in \IN$.
\end{proof}

\begin{lemma}\label{l:ConvergingSubsequences:in:H_n}
 Let $X$ be a sequentially Hausdorff sequential space.
 Let $(H_n)_n$ converge to $H_\infty$ in $\calOOp(X)$ and
 let $(x_n)_n$ be a sequence such that $x_n \in \bigcap(H_n)$ for all $n \in \IN$.
 Then $(x_n)_n$ has a subsequence converging to some element $x_\infty \in \bigcap(H_\infty)$.
\end{lemma}

\begin{proof}
 Choose some open set $U$ contained in $H_\infty$.
 For any $n \in \IN$ the open set $U_n:= X \setminus \{x_n\}=X \setminus \Cls\{x_n\}$ is not in $H_n$,
 because $x_n \in \bigcap(H_n)$.
 By Proposition~\ref{p:calO}\eqref{i:calO:convrel} applied to $X':=\calO(X)$,
 $(U_n)_n$ does not converge to $U$ in $\calO(X)$, 
 because $(H_n)_n$ converges to $H_\infty$ in $\calO( \calO(X) )$.
 By Lemma~\ref{l:TheConvergingSubsequenceLemma}, 
 $(x_n)_n$ has a subsequence $(x_{\varphi(n)})_n$ that converges to some element $x_\infty \in U$.
 Applying Lemma~\ref{l:props:seqHaus}\eqref{i:props:seqHaus:Hnxn} to $(H_{\varphi(n)})_n$,  
 we obtain $x_\infty \in \bigcap(H_\infty)$.
\end{proof}

Lemma~\ref{l:ConvergingSubsequences:in:H_n} implies that $H \mapsto \bigcap H$ 
maps non-empty $\omega$-Scott-open sets to sequentially compact sets 
(in sequentially Hausdorff sequential spaces).
As sequentially compact sets are countably compact (cf.\ \cite{En89}), we obtain:

\begin{corollary}\label{c:bigcap(H):sequentiallycompact}
 Let $X$ be a sequentially Hausdorff sequential space. 
 Then for every non-empty $\omega$-Scott-open collection $H$ of opens
 the intersection $\bigcap(H)$ is both sequentially compact and countably compact.
\end{corollary}

The next example shows that sequential Hausdorffness is essential 
in Corollary~\ref{c:bigcap(H):sequentiallycompact},
even in the case of countably based, locally compact $T_1$-spaces.

\begin{example}\label{ex:seqHaus:essential} \rm
 We consider the space $X$ which has $\IN \cup \{\infty_1,\infty_2\}$ as its underlying set
 and whose topology is generated by the basis
 \[ 
  \mathcal{B}:=
  \big\{ \{n\}, \{a,\infty_1 \,|\, a\geq n \}, \{a,\infty_2 \,|\, a\geq n \} \,\big|\, n \in \IN \big\}
  \,.
 \]
 Clearly, the sequence of natural numbers converges both to $\infty_1$ and to $\infty_2$ in~$X$.
 So all sets in the basis $\mathcal{B}$ are compact, implying that $X$ is locally compact
 and
 \[
  H:=\big\{ U \in \calO(X) \,\big|\, \IN \cup \{\infty_1\} \subseteq U \big\}
     \cup 
     \big\{ U \in \calO(X) \,\big|\, \IN \cup \{\infty_2\} \subseteq U \big\}
 \]
 is an $\omega$-Scott-open family of opens.
 But $\bigcap(H)$ is not countably compact by being equal to the set $\IN$. $\Box$
\end{example}

Now we show that $\bigcap$ is even continuous in our case.

\begin{lemma}\label{l:bigcap:continuous}
 Let $X$ be a sequentially Hausdorff sequential space. 
 Then the map $\mathalpha{\bigcap}\colon \calOOp(X) \to \KuV(X)$, 
 $H \mapsto \bigcap(H)$, is continuous.
\end{lemma}

\begin{proof}
 Suppose the contrary.
 Since $\calOOp(X)$ is sequential, $\bigcap$ is not sequentially continuous,
 hence there is a sequence $(H_n)_n$ converging to $H_\infty$ in $\calOOp(X)$ 
 and an open set $U$ such that $\bigcap(H_\infty) \in \Box U$, but $\bigcap(H_n) \notin \Box U$ for all $n \in \IN$.
 Therefore there exists some $x_n \in \bigcap(H_n) \setminus U$ for all $n \in \IN$.
 By Lemma~\ref{l:ConvergingSubsequences:in:H_n}, $(x_n)_n$ has a subsequence $(x_{\varphi(n)})_n$
 converging to some element $x_\infty \in \bigcap(H_\infty)$.
 Since $x_\infty \in U$, $(x_{\varphi(n)})_n$ is eventually in $U$, a contradiction.
\end{proof}

Now it is easy to prove Theorem~\ref{th:K:retract:OO}. 

\smallskip

\begin{proof} (Theorem~\ref{th:K:retract:OO})
 By Lemmas~\ref{l:K:embeds:into:OO} and~\ref{l:bigcap:continuous},
 the maps $e\colon K \mapsto \{U \in \calO(X) \,|\, K \subseteq U \}$
 and $\mathalpha{\bigcap}\colon H \mapsto \bigcap(H)$ are continuous.
 Let $K \in \KuV(X)$.
 Clearly, $K \subseteq \bigcap( e(K) )$.
 To show the reverse inclusion, let $x \in \bigcap( e(K) )$. 
 Then $X \setminus \{x\} \notin e(K)$.
 Since $X$ is sequentially Hausdorff, $X \setminus \{x\}$ is open,
 hence $K \nsubseteq X \setminus \{x\}$ and $x \in K$.
 We conclude $\bigcap( e(K) ) = K$.
 Therefore $\KuV(X)$ is a continuous retract of $\calOOp(X)$
 with $\bigcap$ as a continuous retraction to the section $e$.
\end{proof}

Theorem~\ref{th:K:retract:OO} implies the following two corollaries.

\begin{corollary}\label{c:calOOp->omegaSOFilter}
 Let $X$ be sequentially Hausdorff and sequential.
 Then there is a continuous retraction 
 from the family of all non-empty $\omega$-Scott-open subsets of $\calO(X)$
 to the family of all non-empty $\omega$-Scott-open filters of $\calO(X)$.
\end{corollary}

As in qcb-spaces compactness and countable compactness agree and
$\omega$-Scott-open sets are Scott-open in the open set lattice, we get:

\begin{corollary}\label{c:qcb:bigcapH:compact}
 Let $X$ be a sequentially Hausdorff qcb-space.
 Then $\bigcap(H)$ is compact for every non-empty Scott-open subset $H \subseteq \calO(X)$.
\end{corollary}


\section*{Acknowledgement}
I thank Matthew de Brecht for posing me the question. 
I had fruitful discussions about the subject with Victor Selivanov and Matthew de Brecht.

\end{document}